\documentclass[12pt]{amsart}
\usepackage{a4}
\usepackage{amssymb}

\newtheorem{theorem}{Theorem}[section]
\newtheorem{Prop}[theorem]{Proposition}
\newtheorem{Thm}[theorem]{Theorem}
\newtheorem{Lem}[theorem]{Lemma}
\newtheorem{Cor}[theorem]{Corollary}

\newtheorem{Conj}[theorem]{Conjecture}

\theoremstyle{definition}

\newtheorem{Def}[theorem]{Definition}
\newtheorem{Ex}[theorem]{Example}
\theoremstyle{remark}

\numberwithin{equation}{section}

\newcommand{\R}{{\mathbb R}}
\newcommand{\C}{{\mathbb C}}
\newcommand{\qu}{{\mathbb H}}
\newcommand{\cay}{{\mathbb O}}
\newcommand{\Z}{{\mathbb Z}}

\newcommand{\Aut}{\mbox{\rm Aut}}
\newcommand{\id}{\mbox{\rm id}}

\begin{document}

\title{Two-point homogeneous quandles with prime cardinality}
\author{Hiroshi Tamaru}
\address{Department of Mathematics, Hiroshima University, 
Higashi-Hiroshima 739-8526, Japan}
\email{tamaru@math.sci.hiroshima-u.ac.jp}
\thanks{A part of this research was supported by KAKENHI (20740040, 24654012).}


\begin{abstract}
Quandles can be regarded as generalizations of symmetric spaces. 
Among symmetric spaces, 
two-point homogeneous Riemannian manifolds would be the most fundamental ones. 
In this paper, we define two-point homogeneous quandles analogously, 
and classify those with prime cardinality. 
\end{abstract}

\maketitle

\section{Introduction}

A \textit{quandle} is a set $X$ endowed with a binary operator 
$\ast : X \times X \to X$ 
satisfying three axioms, 
derived from the Reidemeister moves of a classical knot. 
Quandles were introduced by Joyce (\cite{Joyce}), 
and have played very important roles, in particular, in knot theory. 
For quandles and their applications, 
we refer to a survey \cite{Carter} and the references therein. 

Quandles are also related to symmetric spaces. 
Recall that a \textit{Riemannian symmetric space} is 
a connected Riemannian manifold $M$ equipped with a symmetry 
$s_x : M \to M$ for each $x \in M$. 
For symmetric spaces, 
we refer to famous textbooks \cite{Helgason, Loos}. 
It was mentioned in \cite{Joyce} 
that every symmetric space is a quandle, 
by defining the binary operator $y \ast x := s_x(y)$. 
From this correspondence, 
quandles can be regarded as 
``discrete symmetric spaces''. 
Our theme is to study quandles from this point of view. 
The purpose of this paper is to give the first step 
toward the theory of discrete symmetric spaces. 

In this paper, 
we define the notion of two-point homogeneous quandles, and study them. 
This notion is a natural analogue of the notion of 
two-point homogeneous Riemannian manifolds, 
which form a very fundamental subclass of symmetric spaces. 
In fact, a Riemannian manifold is two-point homogeneous 
if and only if it is isometric to the Euclidean space 
or a Riemannian symmetric space of rank one 
(see Subsect.\ \ref{subsection:riemannian}). 
Hence, 
in the theory of quandles and discrete symmetric spaces, 
two-point homogeneous quandles would play fundamental roles, 
similar to rank one symmetric spaces in the theory of symmetric spaces. 

The main result of this paper is 
an explicit classification of two-point homogeneous quandles with prime cardinality. 
We prove that, for each prime number $p \geq 3$, 
there exists a one-to-one correspondence between 
isomorphism classes of two-point homogeneous quandles with cardinality $p$ 
and primitive roots modulo $p$. 
As a consequence, 
although the condition of two-point homogeneity seems to be very strong, 
two-point homogeneous quandles with cardinality $p$ do exist 
for every prime number $p \geq 3$. 

This paper is organized as follows. 
In Sect.\ \ref{section:preliminaries}, 
we briefly recall some necessary background 
on two-point homogeneous Riemannian manifolds and quandles. 
In Sect.\ \ref{section:THQ}, 
we define two-point homogeneous quandles and study their properties. 
We also define the notion of quandles of cyclic type, 
which gives a sufficient condition for quandles to be two-point homogeneous. 
A classification of two-point homogeneous quandles with prime cardinality, 
mentioned above, 
is given in Sect.\ \ref{section:prime-cardinality}. 
On the way to our classification, 
we recall linear Alexander quandles $\Lambda_p / (t-a)$ with prime cardinality $p$, 
and determine their inner automorphism groups. 
As an appendix, in Sect.\ \ref{section:appendix}, 
we study whether or not the notion of quandles of cyclic type 
is a necessary condition for the two-point homogeneity. 

The author is deeply grateful to 
Seiichi Kamada for valuable comments and suggestions, 
and to the members of Hiroshima University Topology-Geometry Seminar, 
where he learned many things about knots and quandles. 
The author is also grateful to 
Nobuyoshi Takahashi for kind advice, 
and indebted to the referee for improvements in the exposition. 

\section{Preliminaries} 

\label{section:preliminaries}

In this section we briefly recall some necessary background 
on two-point homogeneous Riemannian manifolds and quandles. 

\subsection{Riemannian homogeneous geometry}

\label{subsection:riemannian}

In this subsection we briefly recall some fundamental facts on 
Riemannian symmetric spaces and two-point homogeneous Riemannian manifolds. 
We refer to \cite{Helgason} for the facts mentioned in this subsection. 

\begin{Def}
A connected Riemannian manifold $(M,g)$ is called a 
\textit{Riemannian symmetric space} 
if, for every $x \in M$, 
there exists an isometry $s_x : M \to M$, 
called the \textit{symmetry} at $x$, such that 
\begin{enumerate}
\item[(i)]
$x$ is an isolated fixed point of $s_x$, 
\item[(ii)]
$s_x^2 = \id_M$ (the identity map). 
\end{enumerate}
\end{Def}

We here recall one basic property, 
which will give a connection from symmetric spaces to quandles. 
Note that the following property is sometimes employed 
as one of axioms to define symmetric spaces. 
See \cite{Loos, Nagano, NaganoTanaka}. 

\begin{Prop}
\label{prop:symmetric-property}
Let $(M,g)$ be a Riemannian symmetric space. 
Then, for every $x,y \in M$, we have $s_x \circ s_y = s_{s_x(y)} \circ s_x$. 
\end{Prop}

Next we recall the definition of two-point homogeneous Riemannian manifolds. 
For a Riemannian manifold $(M,g)$, 
we denote by $\mathrm{Isom}(M,g)$ the isometry group, 
and by $d$ the distance function defined by $g$. 

\begin{Def}
A connected Riemannian manifold $(M,g)$ is said to be 
\textit{two-point homogeneous} 
if, for every 
$(x_1, x_2), (y_1, y_2) \in M \times M$ 
satisfying $d(x_1, x_2) = d(y_1, y_2)$, 
there exists $f \in \mathrm{Isom}(M,g)$ 
such that $f(x_1 , x_2) = (y_1 , y_2)$. 
\end{Def}

Note that $f(x_1 , x_2) = (y_1 , y_2)$ 
means that $f(x_1) = y_1$ and $f(x_2) = y_2$. 
We use this notation for simplicity. 

Now we recall a nice characterization and a classification of 
two-point homogeneous Riemannian manifolds. 
We refer to \cite[Ch.\ IX, \S 5]{Helgason} and the references therein. 
Let $\mathrm{Isom}(M,g)_x$ be the isotropy subgroup 
of $\mathrm{Isom}(M,g)$ at $x$, 
which naturally acts on the tangent space $T_x M$. 

\begin{Thm}
\label{thm:helgason}
For a connected Riemannian manifold $(M,g)$, 
the following conditions are mutually equivalent$:$ 
\begin{enumerate}
\item
$(M,g)$ is two-point homogeneous, 
\item
$(M,g)$ is isotropic, 
that is, 
for every $x \in M$, 
the action of $\mathrm{Isom}(M,g)_x$ on the unit sphere in $T_x M$ is transitive, 
\item
$(M,g)$ is isometric to the Euclidean space $\R^n$ 
or a Riemannian symmetric space of rank one. 
\end{enumerate}
\end{Thm}

The action of $\mathrm{Isom}(M,g)_x$ on $T_x M$ is called the 
\textit{isotropy representation} at $x$. 
It is important that the two-point homogeneity can be characterised 
in terms of the isotropy representations. 
We will have an analogous characterization for two-point homogeneous quandles. 

Recall that a Riemannian symmetric space of rank one is either the sphere $S^n$, 
the projective spaces 
$\R \mathrm{P}^n$, $\C \mathrm{P}^n$, $\qu \mathrm{P}^n$, $\cay \mathrm{P}^2$, 
or the hyperbolic spaces 
$\R \mathrm{H}^n$, $\C \mathrm{H}^n$, $\qu \mathrm{H}^n$, $\cay \mathrm{H}^2$. 
Therefore, two-point homogeneous Riemannian manifolds 
form a very fundamental subclass of the class of Riemannian symmetric spaces. 

\subsection{Quandles} 

In this subsection, 
we recall some fundamental notions and examples of quandles. 
We employ a formulation similar to symmetric spaces, 
but start from the usual definition. 

\begin{Def}
\label{def:quandle}
Let $X$ be a set and $\ast : X \times X \to X$ be a binary operator. 
The pair $(X , \ast)$ is called a 
\textit{quandle}
if 
\begin{enumerate}
\item[(Q1)]
$\forall x \in X$, $x \ast x = x$, 
\item[(Q2)]
$\forall x , y \in X$, 
$\exists! z \in X$ : $z \ast y = x$, and 
\item[(Q3)]
$\forall x , y , z \in X$, 
$(x \ast y) \ast z = (x \ast z) \ast (y \ast z)$. 
\end{enumerate}
\end{Def}

If $(X , \ast)$ is a quandle, 
then $\ast$ is called a \textit{quandle structure} on $X$. 
We restate this definition in a similar way to symmetric spaces. 

\begin{Prop}
\label{prop:s-quandle-structure}
Let $X$ be a set, and 
assume that there exists a map $s_x : X \to X$ for every $x \in X$. 
Then, the binary operator $\ast$ defined by 
$y \ast x := s_x(y)$ 
is a quandle structure on $X$ if and only if 
\begin{enumerate}
\item[(S1)]
$\forall x \in X$, $s_x(x) = x$, 
\item[(S2)]
$\forall x \in X$, $s_x$ is bijective, and 
\item[(S3)]
$\forall x , y \in X$, $s_x \circ s_y = s_{s_x(y)} \circ s_x$. 
\end{enumerate}
\end{Prop}

The proof is easy from Definition \ref{def:quandle}. 
Note that, for quandles, $s_x^2 = \id_X$ is not required. 
Throughout this paper, we denote quandles by $X = (X,s)$ 
with the quandle structures 
\begin{align}
s : X \to \mathrm{Map} (X,X) : x \mapsto s_x . 
\end{align}
Here, $\mathrm{Map} (X,X)$ denotes the set of all maps from $X$ to $X$. 
From our formulation, one easily has the following 
(see Proposition \ref{prop:symmetric-property}). 

\begin{Prop}[\cite{Joyce}]
Every Riemannian symmetric space is a quandle. 
\end{Prop}

We here describe some other easy examples of quandles. 

\begin{Ex}
\label{ex:dihedral}
The following $(X,s)$ are quandles: 
\begin{enumerate}
\item
The \textit{trivial quandle}: 
$X$ is any set and $s_x := \id_X$ for every $x \in X$. 
\item
The \textit{dihedral quandle of order $n$}: 
$X = \{ 0, 1 , 2 , \ldots , n-1 \}$ and 
\begin{align*}
s_i(j) := 2i-j \quad (\mathrm{mod} \, n) . 
\end{align*}
\item
The \textit{regular tetrahedron quandle}: 
$X = \{ 1,2,3,4 \}$ 
and 
\begin{align*}
s_1 := (234) , \ \ 
s_2 := (143) , \ \ 
s_3 := (124) , \ \ 
s_4 := (132) . 
\end{align*}
\end{enumerate}
\end{Ex}

Note that $(234)$, $(143)$, and so on, denote the cyclic permutations. 
The dihedral quandle of order $n$ can be realized geometrically 
as the set of $n$-equal dividing points on the unit circle $S^1$. 
The quandle structure is nothing but the restriction 
of the canonical symmetry of $S^1$, 
that is, $s_x$ is the reflection with respect to the line 
through $x$ and the center of $S^1$. 
The regular tetrahedron quandle can be realized geometrically 
as the set of vertices of the regular tetrahedron. 

Next, we recall some fundamental notions for quandles, 
such as homomorphisms, (inner) automorphisms, and connectedness. 

\begin{Def}
Let $(X,s^X)$, $(Y,s^Y)$ be quandles. 
A map $f : X \to Y$ is called a 
\textit{homomorphism} 
if, for every $x \in X$, 
$f \circ s^X_x = s^Y_{f(x)} \circ f$ holds. 
\end{Def}

As usual, a bijective homomorphism is called an \textit{isomorphism}, 
and an isomorphism from $X$ onto $X$ itself is called 
an \textit{automorphism} of $X$. 
By definition, one can see that the automorphism group, denoted by 
\begin{align}
\Aut(X,s) := \{ f : X \to X : \mbox{automorphism} \} , 
\end{align}
is a group. 
It follows from (S2) and (S3) that $s_x \in \Aut(X,s)$ for every $x \in X$. 

\begin{Def}
The subgroup of $\Aut(X,s)$ generated by $\{ s_x \mid x \in X \}$ 
is called the \textit{inner automorphism group} of a quandle $(X,s)$, 
and denoted by $\mathrm{Inn}(X,s)$. 
\end{Def}

Note that $\mathrm{Inn}(X,s)$ and $\Aut(X,s)$ are different in general. 
For example, if $X$ is a trivial quandle, 
then $\mathrm{Inn}(X,s)$ consists of only the identity, 
but $\Aut(X,s)$ coincides with the group of all bijections. 

\begin{Def}
A quandle $(X,s)$ is said to be \textit{connected} if 
$\mathrm{Inn}(X,s)$ acts transitively on $X$. 
\end{Def}

The following gives some easy examples of connected quandles. 
The results are well-known, and the proofs are not difficult. 
Note that $\# X$ denotes the cardinality of $X$. 

\begin{Ex}
\label{ex:connected}
We have the following properties$:$ 
\begin{enumerate}
\item
the trivial quandle $(X,s)$ is not connected if $\# X > 1$, 
\item
the dihedral quandle $(X,s)$ is connected if and only if $\# X$ is odd, 
\item
the regular tetrahedron quandle is connected. 
\end{enumerate}
\end{Ex}

Finally in this subsection, we recall the notion of dual quandles. 

\begin{Def}
Let $(X,s)$ be a quandle. 
Then, $(X,s^{-1})$ defined by the following is called the \textit{dual quandle}: 
\begin{align}
s^{-1} : X \to \mathrm{Map}(X,X) : x \mapsto s_x^{-1} . 
\end{align}
\end{Def}

It is easy to see that the dual quandle $(X,s^{-1})$ is also a quandle. 
A quandle and its dual quandle share many properties, 
as we will see in the latter sections. 
The reason is the following, 
which can be proved directly from the definition. 

\begin{Prop}
\label{prop:dual-inn}
$\mathrm{Inn}(X,s) = \mathrm{Inn}(X,s^{-1})$. 
\end{Prop}

\begin{Def}
A quandle $(X,s)$ is said to be \textit{self-dual} if 
it is isomorphic to the dual quandle $(X,s^{-1})$. 
\end{Def}

\begin{Ex}
The trivial quandles, 
the dihedral quandles, 
and the regular tetrahedron quandle are self-dual. 
\end{Ex}

\section{Two-point homogeneous quandles} 

\label{section:THQ}

In this section, 
we define two-point homogeneous quandles, 
and give a characterization and a sufficient condition 
for quandles to be two-point homogeneous. 
We denote by $G := \mathrm{Inn}(X,s)$ the inner automorphism group 
of a quandle $X = (X,s)$. 

\subsection{Definition}

In this subsection, we define two-point homogeneous quandles, 
analogously to two-point homogeneous Riemannian manifolds. 

\begin{Def}
A quandle $X$ is said to be 
\textit{two-point homogeneous} 
if for all $(x_1 , x_2), (y_1 , y_2) \in X \times X$ 
satisfying $x_1 \neq x_2$ and $y_1 \neq y_2$, 
there exists $f \in G$ such that $f(x_1 , x_2) = (y_1 , y_2)$. 
\end{Def}

It would be more exact to call it 
\textit{two-point homogeneous with respect to the inner automorphism group $G$}. 
One can naturally think of a two-point homogeneous quandle 
with respect to the automorphism group $\Aut(X)$. 
These two notions are different. 
For example, a trivial quandle $X$ with $\# X > 1$ is always 
two-point homogeneous with respect to $\mathrm{Aut}(X)$, 
but not two-point homogeneous with respect to $G$. 

Since the two-point homogeneity is defined in terms of the action of $G$, 
one can immediately see the following. 

\begin{Prop}
\label{prop:TPH-quandle}
If a quandle $(X,s)$ is two-point homogeneous, 
then so is the dual quandle $(X, s^{-1})$. 
\end{Prop}

\begin{proof}
This follows easily from Proposition \ref{prop:dual-inn}. 
\end{proof}

\subsection{A characterization}

Recall that two-point homogeneous Riemannian manifolds 
can be characterized in terms of the isotropy representations 
(see Theorem \ref{thm:helgason}). 
In this subsection, 
we give a similar characterization for two-point homogeneous quandles. 
Let $G_x$ be the isotropy subgroup of $G$ at $x \in X$, that is, 
\begin{align}
G_x := \{ f \in G \mid f(x)=x \} . 
\end{align}

\begin{Prop}
\label{prop:isotropic}
Let $X$ be a quandle and assume that $\# X \geq 3$. 
Then the following conditions are mutually equivalent$:$ 
\begin{enumerate}
\item
$X$ is two-point homogeneous, 
\item
for every $x \in X$, the action of $G_x$ on $X \setminus \{ x \}$ is transitive, 
\item
$X$ is connected, and there exists $x \in X$ 
such that the action of $G_x$ on $X \setminus \{ x \}$ is transitive. 
\end{enumerate}
\end{Prop}

\begin{proof}
We show (1) $\Rightarrow$ (2). 
Take any $x \in X$. 
To show the transitivity, 
take any $y_1, y_2 \in X \setminus \{ x \}$. 
Since $x \neq y_1$ and $x \neq y_2$, 
there exists $f \in G$ such that $f(x, y_1) = (x, y_2)$ by (1). 
This means $f \in G_x$ and $f(y_1) = y_2$. 

We show (2) $\Rightarrow$ (3). 
We have only to show that $X$ is connected. 
Take any $x, y \in X$. 
Since $\# X \geq 3$, there exists $z \in X \setminus \{ x,y \}$. 
By (2), 
the action of $G_z$ on $X \setminus \{ z \}$ is transitive. 
Since $x, y \in X \setminus \{ z \}$, 
there exists $f \in G_z$ such that $f(x) = y$. 

We show (3) $\Rightarrow$ (1). 
Take any $(x_1 , x_2), (y_1 , y_2) \in X \times X$ 
satisfying $x_1 \neq x_2$ and $y_1 \neq y_2$. 
By (3), there exists $x \in X$ such that 
the action of $G_x$ on $X \setminus \{ x \}$ is transitive. 
Also by (3), 
$X$ is connected, 
and hence there exist $f_1 , f_2 \in G$ such that 
$f_1 (x_1) = x$ and $f_2 (y_1) = x$. 
One now has 
\begin{align}
f_1 (x_1 , x_2) = (x , f_1(x_2)) , \quad 
f_2 (y_1 , y_2) = (x , f_2(y_2)) . 
\end{align}
Note that $f_1(x_2) , f_2(y_2) \in X \setminus \{ x \}$. 
Thus, there exists $f_3 \in G_x$ such that 
\begin{align}
f_3 (f_1(x_2)) = f_2(y_2) . 
\end{align}
One can see that 
$f_2^{-1} \circ f_3 \circ f_1 \in G$ 
maps $(x_1 , x_2)$ onto $(y_1 , y_2)$. 
\end{proof}

Examples of two-point homogeneous quandles will be given later. 
Here we see examples of quandles which are not two-point homogeneous. 
This shows that the condition of the two-point homogeneity is very strong. 

\begin{Ex}
The dihedral quandle of order $n \geq 4$ is not two-point homogeneous. 
\end{Ex}

\begin{proof}
Let $X$ be the dihedral quandle of order $n$, 
which we identify with the set of $n$-equal dividing points 
on the unit circle $S^1$ centered at $(0,0) \in \R^2$. 
Let $\mathrm{O}(2)$ be the orthogonal group, 
which is naturally acting on $S^1$. 
Since $s_x \in \mathrm{O}(2)$ for every $x$, 
we have $G \subset \mathrm{O}(2)$. 

Take $x \in X$. 
The isotropy subgroup at $x$ satisfies 
\begin{align}
G_x \subset \mathrm{O}(2)_x = \{ \id , s_x \} \cong \Z_2 . 
\end{align}
By assumption, 
one has $\# (X \setminus \{ x \}) \geq 3$. 
Therefore, $G_x$ cannot act transitively on $X \setminus \{ x \}$. 
From the characterization given in Proposition \ref{prop:isotropic}, 
$X$ is not two-point homogeneous. 
\end{proof}

\subsection{A sufficient condition}

\label{section:c-type}

From now on we assume that 
a quandle $X = (X,s)$ is finite and satisfies $\# X \geq 3$. 
In this subsection, 
we give a sufficient condition, 
called the cyclic type property, 
for quandles to be two-point homogeneous. 
We also see some basic properties of them. 

\begin{Def}
A quandle $(X,s)$ with $\# X = n \geq 3$ is said to be of 
\textit{cyclic type} 
if, for every $x \in X$, 
$s_x$ acts on $X \setminus \{ x \}$ as a cyclic permutation of order $(n-1)$, 
\end{Def}

Let us denote by $\langle s_x \rangle$ the cyclic group generated by $s_x$. 
The cyclic type property is a sufficient condition for quandles 
to be two-point homogeneous. 

\begin{Prop}
\label{prop:c-TPH}
Every quandle of cyclic type is two-point homogeneous. 
\end{Prop}

\begin{proof}
Let $X$ be a quandle of cyclic type. 
Take any $x \in X$. 
From Proposition \ref{prop:isotropic}, 
we have only to show that 
the action of $G_x$ on $X \setminus \{ x \}$ is transitive. 
By assumption, 
$\langle s_x \rangle$ acts transitively on $X \setminus \{ x \}$. 
Thus, so does $G_x$, since $\langle s_x \rangle \subset G_x$. 
\end{proof}

In the remaining of this subsection, 
we study some properties of quandles of cyclic type. 
The first one is the invariance under duality. 

\begin{Prop}
If $(X,s)$ is of cyclic type, 
then so is the dual quandle $(X,s^{-1})$. 
\end{Prop}

\begin{proof}
Take any $x \in X$. 
By assumption, 
$s_x$ acts on $X \setminus \{ x \}$ as a cyclic permutation of order $(n-1)$. 
Then, so does $s^{-1}_x$. 
\end{proof}

The second property of quandles of cyclic type is a characterization, 
similar to Proposition \ref{prop:isotropic} for two-point homogeneous ones. 
It will be stated after a small lemma. 

\begin{Lem}
\label{lem:conjugate}
Let $X = (X,s)$ be a connected quandle. 
Then, for every $x,y \in X$, 
there exists $f \in G$ such that $s_y = f \circ s_x \circ f^{-1}$. 
\end{Lem}

\begin{proof}
Take any $x,y \in X$. 
Since $X$ is connected, 
there exists $f \in G$ such that $y = f(x)$. 
Since $f$ is an automorphism, we have 
\begin{align}
f \circ s_x = s_{f(x)} \circ f = s_y \circ f . 
\end{align}
This completes the proof. 
\end{proof}

\begin{Prop}
\label{prop:cyclic-isotropic}
Let $X = (X,s)$ be a quandle with $\# X = n \geq 3$. 
Then the following conditions are mutually equivalent$:$ 
\begin{enumerate}
\item
$X$ is of cyclic type, 
\item
$X$ is connected, and 
there exists $x \in X$ such that 
$s_x$ acts on $X \setminus \{ x \}$ as a cyclic permutation of order $(n-1)$, 
\item
$X$ is connected, and 
there exists $x \in X$ such that 
$\langle s_x \rangle$ acts transitively on $X \setminus \{ x \}$. 
\end{enumerate}
\end{Prop}

\begin{proof}
We prove (1) $\Rightarrow$ (2). 
Assume that $X$ is of cyclic type. 
We have only to show that $X$ is connected. 
One knows $X$ is two-point homogeneous from Proposition \ref{prop:c-TPH}, 
and hence connected from Proposition \ref{prop:isotropic}. 

We prove (2) $\Rightarrow$ (1). 
Assume that $X$ is connected, and 
there exists $x \in X$ such that 
$s_x$ acts on $X \setminus \{ x \}$ as a cyclic permutation of order $(n-1)$. 
Take any $y \in X$. 
Since $X$ is connected, 
Lemma \ref{lem:conjugate} yields that $s_y$ is conjugate to $s_x$. 
Thus, $s_y$ acts as a cyclic permutation of order $(n-1)$, since so does $s_x$. 
Hence $X$ is of cyclic type. 

The equivalence (2) $\Leftrightarrow$ (3) is obvious. 
\end{proof}

We here see some easy examples of quandles of cyclic type. 
They thus give examples of two-point homogeneous quandles. 
Further examples will be given in the next section. 

\begin{Ex}
The following quandles are of cyclic type, and hence two-point homogeneous: 
\begin{enumerate}
\item
the dihedral quandle of order $3$, and 
\item
the regular tetrahedron quandle. 
\end{enumerate}
\end{Ex}

\begin{proof}
Recall that both quandles are connected (see Example \ref{ex:connected}). 
Thus, we have only to study the action of $s_x$ 
on $X \setminus \{ x \}$ for some $x$. 

Let $X^3 = \{ 0,1,2 \}$ be the dihedral quandle of order $3$. 
For $x = 0$, 
one has $s_0 = (12)$, 
which acts on $X^3 \setminus \{ 0 \} = \{ 1,2 \}$ as a cyclic permutation 
of order $2$. 

Let $X^4 = \{ 1,2,3,4 \}$ be the regular tetrahedron quandle. 
For $x=1$, 
one has $s_1 = (234)$, 
which acts on $X^4 \setminus \{ 1 \} = \{ 2 , 3 , 4 \}$ 
as a cyclic permutation of order $3$. 
\end{proof}

\section{Classification for prime cardinality case}

\label{section:prime-cardinality}

In this section we classify two-point homogeneous quandles 
with prime cardinality $p \geq 3$. 
The proof is based on the classification of 
connected quandles with prime cardinality, 
namely, they must be linear Alexander quandles $\Lambda_p / (t-a)$. 
We determine the inner automorphism groups of $\Lambda_p / (t-a)$, 
and apply it to our classification. 

\subsection{Review on linear Alexander quandles}

\label{subsection:classification1}

In this subsection, 
we review some known results on linear Alexander quandles. 
In fact, every connected quandle with prime cardinality is 
isomorphic to some linear Alexander quandle. 

\begin{Def}
Let $\Lambda_n := \Z_n [t^{\pm 1}]$ be the Laurant polynomial ring over 
$\Z_n = \Z / n \Z$. 
The quotient $\Lambda_n / J$ by an ideal $J$ equipped with the 
following operator is called the \textit{Alexander quandle}: 
\begin{align}
s_{[x]}([y]) := [ty + (1-t)x] . 
\end{align}
\end{Def}

It is easy to check that $\Lambda_n / J$ is a quandle, 
that is, the above operator satisfies the conditions (S1), (S2) and (S3). 

\begin{Def}
Let $a \in \Z$ and $J := (t-a)$, the ideal generated by $t-a$ in $\Lambda_n$. 
Then, the Alexander quandle of the form $\Lambda_n / (t-a)$ is said to be 
\textit{linear}. 
\end{Def}

For a linear Alexander quandle, 
there is a natural identification $\Lambda_n / (t-a) = \Z_n$. 
The quandle operation can be written, 
in terms of the addition and the multiplication on $\Z_n$, as follows: 
\begin{align}
\label{eq:Alexander-operation}
s_{[x]}([y]) := [ty + (1-t)x] 
= [a] [y] + [1-a] [x] . 
\end{align}

\begin{Ex}
The linear Alexander quandle $\Lambda_n / (t-1)$ is trivial, 
and $\Lambda_n / (t+1)$ is isomorphic to the dihedral quandle of order $n$. 
\end{Ex}

\begin{proof}
We only prove the second assertion. 
Consider $\Lambda_n / (t+1)$, that is, $a = n-1$. 
Then, (\ref{eq:Alexander-operation}) yields that 
\begin{align}
s_{[x]} ([y]) = [n-1] [y] + [1 - (n-1)] [x] = [2x-y] . 
\end{align}
This shows that $\Lambda_n / (t+1)$ is the dihedral quandle 
(see Example \ref{ex:dihedral}). 
\end{proof}

From now on we concern with the case $n = p$, a prime number. 
We recall some results obtained by Nelson (\cite{Nelson}). 

\begin{Prop}[\cite{Nelson}]
\label{prop:Nelson}
Let $p$ be a prime number. 
Then we have the following$:$ 
\begin{enumerate}
\item
There are exactly $p-1$ distinct linear Alexander quandles with cardinality $p$ 
up to isomorphism. 
They are $\Lambda_p / (t-a)$ for $a = 1, 2, \ldots, p-1$. 
\item
If $a = 2, \ldots, p-1$, then $\Lambda_p / (t-a)$ is connected. 
\item
$\Lambda_p / (t-a)$ is dual to $\Lambda_p / (t-b)$ 
if and only if 
$ab \equiv 1$ $(\mathrm{mod} \, p)$. 
\end{enumerate}
\end{Prop}

Now we recall the classification of connected quandles with prime cardinality 
obtained in \cite{EGS}. 
We also refer to \cite[Section 5]{Ohtsuki}. 

\begin{Thm}[\cite{EGS}]
\label{thm:classification-prime}
Every connected quandle with prime cardinality $p$ 
is isomorphic to a linear Alexander quandle $\Lambda_p / (t-a)$ 
for some $a = 2, 3, \ldots, p-1$. 
\end{Thm}

Recall that two-point homogeneous quandles must be connected 
(Proposition \ref{prop:isotropic}). 
Therefore, for the classification of prime cardinality ones, 
we have only to determine which linear Alexander quandles are 
two-point homogeneous or not. 

\subsection{The inner automorphism groups of linear Alexander quandles}

\label{subsection:classification2}

Let $X = \Lambda_p / (t-a)$ be a linear Alexander quandle 
with prime cardinality $p$, 
where $a = 2, 3, \ldots, p-1$. 
In this subsection, 
we determine the inner automorphism groups $G = \mathrm{Inn}(X)$ of $X$. 
We identify $X = \Z_p$ as in the previous subsection. 

Recall that $G$ is the group generated by $\{ s_{[x]} \mid [x] \in X \}$. 
First of all, we see formulas for the compositions of generators. 

\begin{Lem}
\label{lem:inner-1}
For $[x], [x_1], [x_2] \in X$ and $k, k_1, k_2 \in \Z$, we have 
\begin{align}
\label{eq:1-inner-1}
(s_{[x]})^k ([y]) 
& = [a^k] [y] + [1-a^k] [x] , \\ 
\label{eq:2-inner-1}
(s_{[x_1]})^{k_1} (s_{[x_2]})^{k_2} ([y]) 
& = [a^{k_1 + k_2}] [y] + [a^{k_1}] [1-a^{k_2}] [x_2] + [1-a^{k_1}] [x_1] . 
\end{align}
\end{Lem}

\begin{proof}
Recall that $s_{[x]}$ is given by (\ref{eq:Alexander-operation}). 
Then, one can show (\ref{eq:1-inner-1}) by induction. 
The proof of (\ref{eq:2-inner-1}) easily follows from (\ref{eq:1-inner-1}). 
\end{proof}

Using these formulas, 
we see that $G$ contains some particular transformations. 
For $[m] \in X$, let us define 
\begin{align}
\varphi_{[m]} : X \to X : [x] \mapsto [x+m] . 
\end{align}

\begin{Lem}
\label{lem:inner-2}
For every $[m] \in X$, we have $\varphi_{[m]} \in G$. 
\end{Lem}

\begin{proof}
Since $\varphi_{[m]} = (\varphi_{[1]})^m$, 
we have only to show that $\varphi_{[1]} \in G$. 
This follows from the following claim: 
\begin{align}
\label{eq:claim-inner-2}
(s_{[1-a]^{-1}}) (s_{[0]})^{p-2} = \varphi_{[1]} . 
\end{align}
Note that $[1-a]$ is invertible. 
We show (\ref{eq:claim-inner-2}). 
Take any $[y] \in X$. 
It follows from Lemma \ref{lem:inner-1} and $[a^{p-1}] = [1]$ that 
\begin{align}
(s_{[1-a]^{-1}}) (s_{[0]})^{p-2} ([y]) 
= [a^{p-1}] [y] + [1-a] [1-a]^{-1} 
= [y] + [1] 
= \varphi_{[1]}([y]) . 
\end{align}
This shows (\ref{eq:claim-inner-2}), 
and hence completes the proof of the lemma. 
\end{proof}

Note that 
$\{ \varphi_{[m]} \mid [m] \in X \}$ 
is a subgroup of $G$, 
and obviously acts transitively on $X$. 
This immediately yields that $X$ is connected. 
Hence, Lemma \ref{lem:inner-2} 
gives a simple proof of Proposition \ref{prop:Nelson} (2). 

We now determine the inner automorphism groups $G$, 
by giving explicit expressions 
in terms of $s_{[x]}$ and $\varphi_{[m]}$

\begin{Thm}
\label{thm:inner-auto-Alexander}
The inner automorphism group of $X = \Lambda_p / (t-a)$ satisfies 
\begin{align}
\label{eq:inner-Alexander}
G = \{ (s_{[x]})^k \mid [x] \in X , \, k \in \Z \} 
\cup \{ \varphi_{[m]} \mid [m] \in X \} . 
\end{align}
\end{Thm}

\begin{proof}
We denote by $(\mathrm{R})$ the 
right-hand side of (\ref{eq:inner-Alexander}) for simplicity. 
One knows $G \supset (\mathrm{R})$ from Lemma \ref{lem:inner-2}. 
We prove $G \subset (\mathrm{R})$. 
Note that $(\mathrm{R})$ contains generators of $G$, that is, 
\begin{align}
\{ s_{[x]} \mid [x] \in X \} \subset (\mathrm{R}) . 
\end{align}
Therefore, it is enough to show that $(\mathrm{R})$ is a group. 
To show this, one has only to check 
\begin{align}
(s_{[x_1]})^{k_1} (s_{[x_2]})^{k_2} , \ 
(s_{[x]})^{k} (\varphi_{[m]}) , \ 
(\varphi_{[m]}) (s_{[x]})^{k} \in (\mathrm{R}) . 
\end{align}

Claim 1: 
$(s_{[x_1]})^{k_1} (s_{[x_2]})^{k_2} \in (\mathrm{R})$. 
For simplicity, let 
\begin{align}
[z] := [a^{k_1}] [1-a^{k_2}] [x_2] + [1-a^{k_1}] [x_1] . 
\end{align}
Then, Lemma \ref{lem:inner-1} yields that 
\begin{align}
(s_{[x_1]})^{k_1} (s_{[x_2]})^{k_2} ([y]) = [a^{k_1 + k_2}] [y] + [z] . 
\end{align}
Case 1: Assume that $[a^{k_1 + k_2}] = [1]$. 
In this case, we have 
\begin{align}
(s_{[x_1]})^{k_1} (s_{[x_2]})^{k_2} ([y]) 
= [y] + [z] 
= \varphi_{[z]} ([y]) . 
\end{align}
This yields that 
\begin{align}
\begin{split}
(s_{[x_1]})^{k_1} (s_{[x_2]})^{k_2} 
= \varphi_{[z]} \in (\mathrm{R}) . 
\end{split}
\end{align}
Case 2: Assume that $[a^{k_1 + k_2}] \neq [1]$. 
In this case, one has 
\begin{align}
\begin{split}
(s_{[x_1]})^{k_1} (s_{[x_2]})^{k_2} ([y]) 
& = [a^{k_1 + k_2}] [y] + [1 - a^{k_1 + k_2}] [1 - a^{k_1 + k_2}]^{-1} [z] \\ 
& = (s_{[1 - a^{k_1 + k_2}]^{-1} [z]})^{k_1 + k_2} ([y]) . 
\end{split}
\end{align}
This yields that 
\begin{align}
\begin{split}
(s_{[x_1]})^{k_1} (s_{[x_2]})^{k_2} 
= (s_{[1 - a^{k_1 + k_2}]^{-1} [z]})^{k_1 + k_2} \in (\mathrm{R}) . 
\end{split}
\end{align}
From the arguments in Case 1 and Case 2, we conclude Claim 1. 

Claim 2: 
$(s_{[x]})^{k} (\varphi_{[m]}) \in (\mathrm{R})$. 
First of all, one has 
\begin{align}
\begin{split}
(s_{[x]}) (\varphi_{[m]}) ([y]) 
& = s_{[x]} ([y+m]) \\ 
& = [a] [y+m] + [1-a] [x] \\
& = [a] [y] + [1-a] ( [1-a]^{-1} [am] + [x] ) \\
& = s_{[1-a]^{-1} [am] + [x]} ([y]) . 
\end{split}
\end{align}
Hence, Claim 1 yields that 
\begin{align}
\begin{split}
(s_{[x]})^{k} (\varphi_{[m]})
= (s_{[x]})^{k-1} (s_{[1-a]^{-1} [am] + [x]}) \in (\mathrm{R}) . 
\end{split}
\end{align}
This concludes the proof of Claim 2. 

Claim 3: 
$(\varphi_{[m]}) (s_{[x]})^{k} \in (\mathrm{R})$. 
This can be proved similarly to Claim 2. 
\end{proof}

As a corollary, 
we have an explicit expression of the isotropy subgroup $G_{[0]}$ 
of $G$ at $[0] \in X$. 
This expression is crucial for the classification 
of two-point homogeneous quandles. 
Recall that $\langle s_{[0]} \rangle$ denotes the cyclic group 
generated by $s_{[0]}$. 

\begin{Cor}
\label{cor:isotropy}
The inner automorphism group of $X = \Lambda_p / (t-a)$ satisfies 
\begin{align}
G_{[0]} = \langle s_{[0]} \rangle . 
\end{align}
\end{Cor}

\begin{proof}
One knows ($\supset$). 
To prove ($\subset$), it is enough to show that 
\begin{align}
\label{eq:claim-isotropy}
\# G_{[0]} \leq \# \langle s_{[0]} \rangle . 
\end{align}
We show this inequality. 
From Theorem \ref{thm:inner-auto-Alexander}, one has 
\begin{align}
\label{eq:claim10-30}
G = \{ \id \} \cup 
\bigcup_{[x] \in X} (\langle s_{[x]} \rangle \setminus \{ \id \}) 
\cup \{ \varphi_{[m]} \mid [0] \neq [m] \in X \} . 
\end{align}
Since $X$ is connected, 
Lemma \ref{lem:conjugate} yields that $s_{[x]}$ and $s_{[0]}$ are conjugate 
for every $[x] \in X$. 
One thus has 
\begin{align}
\# \langle s_{[x]} \rangle = \# \langle s_{[0]} \rangle . 
\end{align}
This yields that 
\begin{align}
\# G 
\leq 1 + \sum_{[x] \in X} (\# \langle s_{[x]} \rangle -1) + (p-1) 
= \# \langle s_{[0]} \rangle \cdot p . 
\end{align}
Recall that $G$ acts transitively on $X$, and hence $X = G / G_{[0]}$. 
We thus have 
\begin{align}
p = \# X = \# G / \# G_{[0]} 
\leq (\# \langle s_{[0]} \rangle \cdot p) / \# G_{[0]} . 
\end{align}
This proves (\ref{eq:claim-isotropy}), 
and hence completes the proof of the corollary. 
\end{proof}

\subsection{A classification}

\label{subsection:classification3}

In this subsection, 
we classify two-point homogeneous quandles and quandles of cyclic type, 
with prime cardinality $p \geq 3$. 

For the classification, we have only to study linear Alexander quandles. 
We begin with an elementary observation. 

\begin{Lem}
\label{lem:before-main}
Let $X = \Lambda_p / (t-a) = \Z_p$ be a linear Alexander quandle. 
Then we have 
\begin{align}
\langle s_{[0]} \rangle.[1] = \{ [1] , [a] , [a^2] , [a^3] , \ldots \} . 
\end{align}
\end{Lem}

\begin{proof}
By definition, one has 
\begin{align}
s_{[0]} ([y]) = [a] [y] + [1-a] [0] = [ay] . 
\end{align}
Then the proof easily follows. 
\end{proof}

We denote the multiplicative group of $\Z_p$ by 
\begin{align}
(\Z_p)^{\times} := \Z_p \setminus \{ [0] \} . 
\end{align}
One knows that $(\Z_p)^{\times}$ is a cyclic group of order $p-1$. 
An integer $a$ ($2 \leq a \leq p-1$) is called a 
\textit{primitive root modulo $p$} 
if $[a]$ generates $(\Z_p)^{\times}$. 

\begin{Thm}
\label{thm:classification}
Let $X$ be a quandle with prime cardinality $p \geq 3$. 
Then, 
the following conditions are mutually equivalent$:$ 
\begin{enumerate}
\item
$X$ is two-point homogeneous, 
\item
$X$ is isomorphic to the linear Alexander quandle $\Lambda_p / (t-a)$, 
where $a$ is a primitive root modulo $p$, 
\item
$X$ is of cyclic type. 
\end{enumerate} 
\end{Thm}

\begin{proof}
We prove (1) $\Rightarrow$ (2). 
Assume that $X$ is two-point homogeneous. 
Proposition \ref{prop:isotropic} yields that 
$X$ is connected and $G_{[0]}$ acts transitively on $X \setminus \{ [0] \}$. 
Since $X$ is connected, 
it is isomorphic to $\Lambda_p / (t-a)$ for some $a = 2, 3, \ldots, p-1$ 
(see Theorem \ref{thm:classification-prime}). 
Since $G_{[0]}$ acts transitively on $X \setminus \{ [0] \}$, one has 
\begin{align}
G_{[0]}.[1] 
= X \setminus \{ [0] \} 
= (\Z_p)^{\times} . 
\end{align}
Furthermore, 
from Corollary \ref{cor:isotropy} and Lemma \ref{lem:before-main}, 
we have 
\begin{align}
G_{[0]}.[1] 
= \langle s_{[0]} \rangle.[1] 
= \{ [1] , [a] , [a^2] , [a^3] , \ldots \} . 
\end{align}
Therefore, 
$[a]$ generates $(\Z_p)^{\times}$, 
and hence $a$ is a primitive root modulo $p$. 

We prove (2) $\Rightarrow$ (3). 
Let $X := \Lambda_p / (t-a)$ 
and assume that $a$ is a primitive root modulo $p$. 
Since $a \neq 1$, one knows that $X$ is connected 
(see Proposition \ref{prop:Nelson}). 
Thus, from Proposition \ref{prop:cyclic-isotropic}, 
we have only to show that 
$\langle s_{[0]} \rangle$ acts transitively on $X \setminus \{ [0] \}$. 
This follows from Lemma \ref{lem:before-main} and the assumption on $a$, 
\begin{align}
\langle s_{[0]} \rangle.[1] 
= \{ [1] , [a] , [a^2] , [a^3] , \ldots \} 
= (\Z_p)^{\times} = X \setminus \{ [0] \} . 
\end{align}
Hence $X$ is of cyclic type. 

One knows (3) $\Rightarrow$ (1) from Proposition \ref{prop:c-TPH}. 
\end{proof}

Recall that, for every prime number $p \geq 3$, 
there always exists a primitive root $a$ modulo $p$. 
This implies the following. 

\begin{Cor}
For each prime number $p \geq 3$, 
there exists a two-point homogeneous quandle with cardinality $p$. 
\end{Cor}

Furthermore, in general, 
there are several non-isomorphic two-point homogeneous quandles 
with the same cardinality $p$. 
We list the number of two-point homogeneous quandles 
with prime cardinality $p \leq 37$ 
in Table \ref{table}. 
The first column denote the prime number $p$. 
The second column $\#$ represents the number of 
the isomorphism classes of two-point homogeneous quandles 
with cardinality $p$. 
The third column lists $a$ so that 
$\Lambda_p / (t-a)$ is two-point homogeneous, 
that is, $a$ is a primitive root modulo $p$ 
(we refer to \cite{Takagi}). 
Here, $a \leftrightarrow b$ means that 
$\Lambda_p / (t-a)$ and $\Lambda_p / (t-b)$ are dual to each other, 
that is, $ab \equiv 1$ $(\mathrm{mod} \, p)$ 
(see Proposition \ref{prop:Nelson}).

\begin{table}[ht]
\begin{tabular}{c|c|l}
$p$ & $\#$ & $a$ (so that $\Lambda_p / (t-a)$ is two-point homogeneous) \\ \hline 
$3$ & $1$ & $2$ \\ 
$5$ & $2$ & $2 \leftrightarrow 3$ \\ 
$7$ & $2$ & $3 \leftrightarrow 5$ \\ 
$11$ & $4$ & $2 \leftrightarrow 6$, \ $7 \leftrightarrow 8$ \\ 
$13$ & $4$ & $2 \leftrightarrow 7$, \ $6 \leftrightarrow 11$ \\ 
$17$ & $8$ & $3 \leftrightarrow 6$, \ $5 \leftrightarrow 7$, \  
$10 \leftrightarrow 12$, \ $11 \leftrightarrow 14$ \\ 
$19$ & $6$ & $2 \leftrightarrow 10$, \ $3 \leftrightarrow 13$, \ $14 \leftrightarrow 15$ \\
$23$ & $10$ & 
$5 \leftrightarrow 14$, \
$7 \leftrightarrow 10$, \
$11 \leftrightarrow 21$, \
$15 \leftrightarrow 20$, \
$17 \leftrightarrow 19$ \\ 
$29$ & $12$ & 
$2 \leftrightarrow 15$, \
$3 \leftrightarrow 10$, \
$8 \leftrightarrow 11$, \
$14 \leftrightarrow 27$, \
$18 \leftrightarrow 21$, \
$19 \leftrightarrow 26$ \\
$31$ & $8$ & 
$3 \leftrightarrow 21$, \
$11 \leftrightarrow 17$, \
$12 \leftrightarrow 13$, \
$22 \leftrightarrow 24$ \\
$37$ & $12$ & 
$2 \leftrightarrow 19$, \
$5 \leftrightarrow 15$, \
$13 \leftrightarrow 20$, \
$17 \leftrightarrow 24$, \
$18 \leftrightarrow 35$, \
$22 \leftrightarrow 32$. 
\end{tabular}
\medskip 
\caption{The number of two-point homogeneous quandles}
\label{table}
\end{table}

\section{Appendix}

\label{section:appendix}

We saw in Proposition \ref{prop:c-TPH} that, 
if $X$ is of cyclic type, then it is two-point homogeneous. 
We are interested in whether the converse of this statement holds or not. 
Our naive conjecture is the following. 

\begin{Conj}
\label{conjecture}
Let $X$ be a two-point homogeneous quandle with finite cardinality. 
Then $X$ is of cyclic type. 
\end{Conj}

Theorem \ref{thm:classification} yields that, 
our conjecture is true if $\# X$ is prime. 
We show in this Appendix that, 
our conjecture is also true if $\# X - 1$ is prime. 

\begin{Lem}
\label{lem:converse-1}
Let $(X,s)$ be a quandle and $x \in X$. 
Then, $s_x$ is contained in the center of $G_x$. 
\end{Lem}

\begin{proof}
Take any $f \in G_x$. 
Then we have 
\begin{align*}
f \circ s_x = s_{f(x)} \circ f = s_x \circ f , 
\end{align*}
since $f$ is an automorphism and $f(x) = x$. 
\end{proof}

\begin{Prop}
\label{prop:THQ-C}
Let $X$ be a two-point homogeneous quandle, 
and assume that $\# X = p+1$, where $p$ is a prime number. 
Then $X$ is of cyclic type. 
\end{Prop}

\begin{proof}
Let $X$ be a two-point homogeneous quandle with $\# X = p+1$. 
Note that $X$ is connected from Proposition \ref{prop:isotropic}. 
Take any $x \in X$. 
From Proposition \ref{prop:cyclic-isotropic}, 
we have only to show that $\langle s_x \rangle$ acts transitively 
on $X \setminus \{ x \}$. 

Let us consider the orbit decomposition 
\begin{align*}
X \setminus \{ x \} = \coprod_{i=1}^m \langle s_x \rangle . y_i , 
\end{align*}
where $\langle s_x \rangle . y_i$ denotes the $\langle s_x \rangle$-orbit 
through $y_i$, and $m$ is the number of orbits. 
We are going to show that $m=1$. 

Claim 1: 
All $\langle s_x \rangle$-orbits have the same cardinality. 
To show this, take any $y , z \in X \setminus \{ x \}$. 
Since $X$ is two-point homogeneous, 
Proposition \ref{prop:isotropic} yields that 
there exists $f \in G_x$ such that $f(y) = z$. 
From Lemma \ref{lem:converse-1}, we have 
\begin{align*}
f( \langle s_x \rangle . y) 
= \langle s_x \rangle . f(y) 
= \langle s_x \rangle . z . 
\end{align*}
This implies that 
$\langle s_x \rangle . y$ and $\langle s_x \rangle . z$ have the same cardinality. 

Claim 2: $\# ( \langle s_x \rangle . y_1) > 1$. 
Assume that $\# ( \langle s_x \rangle . y_1) = 1$. 
Then, Claim 1 yields that 
every $\langle s_x \rangle$-orbit has the cardinality $1$. 
This means $s_x = \id_X$. 
Since $X$ is connected, 
Lemma \ref{lem:conjugate} yields that $s_z = \id_X$ for all $z \in X$. 
This implies that $X$ is a trivial quandle, 
which contradicts the connectedness of $X$. 

Claim 3: $m=1$. 
From Claim 1, we have 
\begin{align*}
p = \# (X \setminus \{ x \}) = m \cdot \# ( \langle s_x \rangle . y_1) . 
\end{align*}
Since $p$ is a prime number by assumption, Claim 2 yields that $m=1$. 
\end{proof}

\bibliographystyle{amsplain}

\end{document}